\documentclass[12pt,amscd]{amsart}
\usepackage[all]{xy}
\usepackage{graphicx}
\usepackage{amsmath,amsxtra,amssymb,latexsym, amscd,amsthm}

\usepackage{tikz}
\usepackage{ytableau}

 \usepackage{indentfirst}
\usepackage[mathscr]{eucal}
  \usepackage[pagebackref=true]{hyperref}

\newtheorem{thm}{Theorem}[section]
\newtheorem{cor}[thm]{Corollary}
\newtheorem{lem}[thm]{Lemma}

\theoremstyle{definition}
\newtheorem{defn}[thm]{Definition}

\newtheorem{conj}[thm]{Conjecture}

\numberwithin{equation}{section}

\DeclareMathOperator{\depth}{depth}

\DeclareMathOperator{\reg}{reg}

\def\Adm{\operatorname{Adm}}

\def\a {\mathbf a}

\def\k {\mathrm{k}}

\begin{document}

\title{Ordered alternating paths and the depth of symbolic powers of cover ideals of graphs}

\author[N.T. Hang]{Nguyen Thu Hang }
\address{Thai Nguyen University of Sciences, Phan Dinh phung Ward, Thai Nguyen, Vietnam}
\email{hangnt@tnus.edu.vn}

\author[N. T. T. Tam]{Nguyen Thi Thanh Tam}
\address{Hung Vuong University, Phu Tho, Vietnam}
\email{nguyenthithanhtam@hvu.edu.vn}

\author{Thanh Vu}
\address{Institute of Mathematics, VAST, 18 Hoang Quoc Viet, Hanoi, Vietnam}
\email{vuqthanh@gmail.com}

\subjclass[2020]{13D02, 05E40, 13F55}
\keywords{depth; tree; cover ideal; ordered matching; alternating path}

\date{}

\dedicatory{Dedicated to Professor Le Tuan Hoa on the occasion of his 70th birthday}
\commby{}

\begin{abstract} Let $G$ be a simple graph with cover ideal $J(G)$ in a polynomial ring $S$ in $|V(G)|$ variables. For a matching $M$ of $G$, we denote by $\ell(M)$ the length of the longest $M$-alternating path in $G$. We define $\alpha_t(G)$ to be the maximum size of an ordered matching $M$ of $G$ such that $\ell(M) \le 2t-1$. We then prove that 
$$\operatorname{depth}(S/J(G)^{(t)}) \le |V(G)| - 1 - \alpha_t(G)$$
for all $t \ge 1$, where $J(G)^{(t)}$ denotes the $t$-th symbolic power of $J(G)$, and that equality holds when $G$ is a forest.
\end{abstract}

\maketitle
\section{Introduction}
\label{sect_intro}
Alternating paths play a crucial role in matching theory \cite{LP}. Alternating paths with respect to an ordered matching, which we call \emph{ordered alternating paths}, were used to control the stabilization index of the depth of symbolic powers of cover ideals of graphs \cite{BHHT}. In this work, we show that this notion can also be used to control the depth of each symbolic power of the cover ideal of a graph. We now introduce the necessary notation.

Let $G$ be a simple graph. A subset $M \subseteq E(G)$ is called a \emph{matching} of $G$ if no two edges in $M$ share a common vertex. A subset of vertices $U \subseteq V(G)$ is called \emph{independent} if
\( \{u,v\} \notin E(G)
\) for all $u,v \in U$.

\begin{defn}
Let
\[
M = \{ \{ u_1,v_1\}, \ldots, \{u_r,v_r\}\}
\]
be a matching of $G$. The matching $M$ is called an \emph{ordered matching} of $G$ if 
\begin{enumerate}
    \item $\{u_1, \ldots, u_r\}$ is an independent set;
    \item \( \{u_i,v_j\} \in E(G)\) implies that $i \le j$.
\end{enumerate}
The set $A=\{u_1,\ldots,u_r\}$ is called the \emph{free parameter set} of $M$, and the set $B=\{v_1,\ldots,v_r\}$ is called the \emph{partner set} of $M$. The \emph{ordered matching number} of $G$, denoted by $\nu_o(G)$, is the maximum cardinality of an ordered matching of $G$.
\end{defn}

\begin{defn}
Let $G$ be a simple graph, and let $M$ be a matching of $G$. An \emph{$M$-alternating path} is a path whose edges alternately belong to $M$ and $E(G)\setminus M$. Let $\ell(M)$ denote the maximum length of an $M$-alternating path. For each integer $t \ge 1$, we define
\[
\alpha_t(G)
=
\max\left\{
|M|
\;\middle|\;
\text{$M$ is an ordered matching of $G$ and } \ell(M)\le 2t-1
\right\}.
\]
\end{defn}
In general, an $M$-alternating path might not be simple. The length of such a path, by definition, is the number of edges that it traverses.

Let $S=\k[x_1,\dots,x_n]$ be a polynomial ring over a field $\k$, and let $G$ be a simple graph with vertex set $V(G)=[n]=\{1,\ldots,n\}$. Denote by $J(G)$ the cover ideal of $G$. Our main result is the following.

\begin{thm}\label{main}
Let $G$ be a simple graph with vertex set $V(G)=[n]$. Then, for every $t\ge 1$,
\[
\depth\!\left(S/J(G)^{(t)}\right)\le n-1-\alpha_t(G),
\]
where $J(G)^{(t)}$ is the $t$-th symbolic power of $J(G)$. Moreover, equality holds whenever $G$ is a forest.
\end{thm}

The notion of an ordered matching was introduced by Constantinescu and Varbaro \cite{CV}, who proved that 
$$\lim_{t\to \infty} \operatorname{depth}(S/J(G)^{(t)}) = n - 1 - \nu_o(G)$$
for an arbitrary graph $G$. Later, Hoa, Kimura, Terai, and Trung \cite{HKTT} proved that the function $\operatorname{depth}(S/J(G)^{(t)})$ is weakly decreasing. Nonetheless, explicit values for the depth of symbolic powers of $J(G)$ have only recently become known for paths \cite{DHNT}, cycles \cite{DHV}, and Ferrers graphs \cite{HV}. These calculations use a result of Dung, Hang, and Vu \cite{DHV}, which translates the problem of computing the depth of symbolic powers of the cover ideal of $G$ into that of computing the regularity of the edge ideals of admissible subgraphs of $G$. 

In the next section, we recall the necessary notation and then prove Theorem~\ref{main}.

\section{Depth of symbolic powers of cover ideals via admissible subgraphs}
Throughout this paper, let \( \k \) be a field, \( S=\k[x_1,\ldots,x_n] \) a polynomial ring, and \( \mathfrak{m}=(x_1,\ldots,x_n) \) the \emph{homogeneous maximal ideal} of \( S \).

\subsection{Depth and Castelnuovo--Mumford regularity}

Let \(M\) be a nonzero finitely generated graded \(S\)-module. For each integer \(i\ge 0\), let \(H_{\mathfrak{m}}^i(M)\) denote the \(i\)-th local cohomology module of \(M\) with support in \(\mathfrak{m}\). The depth and the Castelnuovo--Mumford regularity of \(M\) are defined by
\[
\depth(M)=\min\{\,i \mid H_{\mathfrak{m}}^i(M)\neq 0\,\},
\]
and
\[
\reg(M)=\max\{\,j+i \mid H_{\mathfrak{m}}^i(M)_j\neq 0,\ 0\le i\le \dim(M),\ j\in\mathbb{Z}\,\}.
\]

\subsection{Graphs and their edge ideals and cover ideals}  Let $G$ be a simple graph with vertex set $V(G) = \{1, \ldots, n\}$ and edge set $E(G)$. Throughout, we assume that $G$ has no isolated vertices.

\begin{defn}\label{definition1}
Let \( G \) be a simple graph with vertex set \( V(G) = \{1, \ldots, n\} \) and edge set \( E(G) \).

\begin{enumerate}
    \item A simple graph \( H \) is a subgraph of \( G \) if \( V(H) \subseteq V(G) \) and \( E(H) \subseteq E(G) \). It is an \emph{induced subgraph} of \( G \) if \( E(H) = E(G) \cap \big( V(H) \times V(H) \big) \).
    
    \item For a subset \( U \subseteq V(G) \), we denote by \( G[U] \) and \( G \setminus U \) the induced subgraphs of \( G \) on \( U \) and on \( V(G) \setminus U \), respectively.
    
    \item A path \( P_n \) on \( n \) vertices is the graph with vertex set \( V(P_n) = \{1, \ldots, n\} \) and edge set
    \[
    E(P_n) = \{ \{1,2\}, \ldots, \{n-1,n\} \}.
    \]
    
    \item A cycle \( C_n \) on \( n \) vertices is the graph with vertex set \( V(C_n) = \{1, \ldots, n\} \) and edge set
    \[
    E(C_n) = E(P_n) \cup \{ \{1,n\} \}.
    \]
    
    \item A \emph{forest} is a graph with no cycles. A \emph{tree} is a connected forest.
    
    \item A subset \( M \subseteq E(G) \) is called a \emph{matching} of \( G \) if no two edges in \( M \) share a common vertex. It is an \emph{induced matching} if the subgraph induced by the vertices of \( M \) has edge set exactly \( M \). The \emph{induced matching number} of \( G \), denoted by $ \nu(G) $, is the maximum size of an induced matching in \( G \).
    
   \item    A subset $W \subseteq V(G)$ is called a \emph{vertex cover} of $G$ if $W \cap e \neq \emptyset$ for every edge $e \in E(G)$. It is called a \emph{minimal vertex cover} if no proper subset of $W$ is a vertex cover of $G$.
\end{enumerate}
\end{defn}

\begin{defn}
Let $G$ be a graph with vertex set $V(G) = \{1, \ldots, n\}$ and edge set $E(G)$. The \emph{edge ideal} and \emph{cover ideal} of $G$, denoted by $I(G)$ and $J(G)$, respectively, are defined by
$$I(G) = (x_ix_j \mid \{i,j\} \in E(G)),
\quad \text{and} \quad
J(G) = \bigcap_{\{i,j\} \in E(G)} (x_i,x_j).$$
\end{defn}

In particular, $J(G)$ is the Alexander dual of $I(G)$. It is well known that
$$J(G) = (x_W \mid W \text{ is a minimal vertex cover of } G),$$
where $x_W = \prod_{i \in W} x_i$. The $t$-th symbolic power of $J(G)$ is given by
$$J(G)^{(t)}=\bigcap_{ \{i,j\} \in E(G)}(x_i, x_j)^t.$$

\subsection{Ordered alternating paths} 
\begin{defn} 
Let $M$ be an ordered matching in $G$ with free parameter set $A$ and partner set $B$. For a vertex $v$ covered by $M$, an $M$-alternating path $p$ is called an $M$-admissible path for $v$ if:
\begin{enumerate}
    \item $v$ is the origin of $p$,
    \item the first and the last edges of $p$ belong to $M$,
    \item every edge of $p$ has one endpoint in $A$ and the other endpoint in $B$.
\end{enumerate}
\end{defn}

\begin{defn} 
Let $M$ be an ordered matching in a graph $G$ with free parameter set $A$ and partner set $B$. For every vertex $v$ covered by $M$, let $\ell(v)$ be the length of a longest $M$-admissible path for $v$. We define
$$\ell_0 (M) = \max \{ \ell (v) \mid v \in B \} \quad \text{and} \quad \ell_1(M) = \max \{ \ell (u) + \ell(v) + 1 \mid \{u,v\} \in E(G[B]) \},$$
where we adopt the convention that $\ell_1(M) = 0$ if $B$ is an independent set in $G$.   
\end{defn}
The following result \cite[Lemma~1.11]{BHHT} shows that, when $M$ is an ordered matching, $\ell(M)$ can be computed in terms of $\ell_0(M)$ and $\ell_1(M)$.

\begin{lem}\label{lem_length} Let $M$ be an ordered matching in a graph $G$. Then 
$$\ell(M) = \max \{ \ell_0(M), \ell_1(M)\}.$$
\end{lem}

\subsection{Admissible subgraphs}
\begin{defn}
Let $G$ be a graph with vertex set $V(G)=[n]$. A subgraph $H$ of $G$ is called a \emph{$t$-admissible subgraph} of $G$ if there exists an exponent vector $\mathbf{a}=(a_1,\ldots,a_n)\in\mathbb{N}^n$ such that
\[
E(H)=\bigl\{\{i,j\}\in E(G)\mid a_i+a_j<t\bigr\}.
\]
The set of all $t$-admissible subgraphs of $G$ is denoted by $\operatorname{Adm}_t(G)$. Such an exponent vector $\mathbf{a}$ is called a \emph{certificate of admissibility} for $H$. The value of $\mathbf{a}$ at a vertex $v\in V(G)$ is also denoted by $a(v)$.
\end{defn}

The following result, due to Dung, Hang, and Vu \cite[Lemma~2.11]{DHV}, is the cornerstone of our approach.

\begin{thm}\label{thm:admissible}
Let $G$ be a simple graph. Then
\[
\depth\!\left(S/J(G)^{(t)}\right)
=
n-
\max\left\{
\reg(I(H))
\;\middle|\;
H \in \Adm_t(G)
\right\}.
\]
\end{thm}

\subsection{Proof of the main result}
We first establish some auxiliary lemmas.

\begin{lem}\label{lem_ind_order} 
Let $G$ be a simple graph and $H \in \operatorname{Adm}_t(G)$. If $M$ is an induced matching of $H$, then $M$ is an ordered matching of $G$.     
\end{lem}

\begin{proof} 
Let $\a \in \mathbb{N}^n$ be a certificate of admissibility for $H$. Let the edges of $M$ be $\{u_1, v_1\}, \ldots, \{u_r, v_r\}$. For each edge $\{u_i, v_i\} \in M$, we choose the orientation such that $a(u_i) \le a(v_i)$. We then order the edges of the matching so that $a(v_1) \le a(v_2) \le \cdots \le a(v_r)$. We will now show that $M$ constitutes an ordered matching under this ordering. 

First, we prove that $A = \{u_1, \ldots, u_r\}$ is an independent set in $G$. Suppose by contradiction that $\{u_i, u_j\} \in E(G)$ for some $i < j$. Since $M$ is an induced matching in $H$, it follows that $\{u_i, u_j\} \notin E(H)$, which implies $a(u_i) + a(u_j) \ge t$. On the other hand, by the admissibility of $H$, we have $a(u_i) + a(v_i) \le t-1$ and $a(u_j) + a(v_j) \le t-1$. Summing these inequalities yields $a(u_i) + a(u_j) + a(v_i) + a(v_j) \le 2t-2$. Given our assumption that $a(u_i) \le a(v_i)$ and $a(u_j) \le a(v_j)$, this implies $2a(u_i) + 2a(u_j) \le 2t-2$, or equivalently, $a(u_i) + a(u_j) \le t-1$. This contradicts our earlier deduction that $a(u_i) + a(u_j) \ge t$. Hence, $A$ must be an independent set in $G$. 

Next, assume that $\{u_i, v_j\} \in E(G)$ for some $i \neq j$. Since $M$ is an induced matching in $H$, we have $\{u_i, v_j\} \notin E(H)$, from which we deduce that $a(u_i) + a(v_j) \ge t > a(u_i) + a(v_i)$. This strictly implies $a(v_j) > a(v_i)$. By our chosen ordering of the vertices, it follows that $i < j$, completing the proof.   
\end{proof}

\begin{lem}\label{lem:1}
Let $G$ be a simple graph and $H \in \operatorname{Adm}_t(G)$. If $M$ is an induced matching of $H$, then $\ell(M) \le 2t-1$.
\end{lem}

\begin{proof} 
As in the proof of Lemma \ref{lem_ind_order}, let $\a \in \mathbb{N}^n$ be a certificate of admissibility for $H$, and let the edges of $M$ be $\{u_1, v_1\}, \ldots, \{u_r, v_r\}$ ordered such that $a(u_i) \le a(v_i)$ and $a(v_1) \le a(v_2) \le \cdots \le a(v_r)$. 

For each $i = 1, \ldots, r$, denote by $d_i$ the maximum number of matching edges in an $M$-admissible path for $v_i$. In other words, $\ell(v_i) = 2d_i - 1$. By Lemma~\ref{lem_length}, we have $\ell(M) = \max \{ \ell_0(M), \ell_1(M) \}$. We will now show that both $\ell_0(M)$ and $\ell_1(M)$ are at most $2t-1$.

First, consider an $M$-admissible path $p = v_{i_1}, u_{i_1}, v_{i_2}, u_{i_2}, \ldots, v_{i_s}, u_{i_s}$. Since $M$ is an induced matching in $H$, the edge $\{u_{i_j}, v_{i_{j+1}}\}$ is not in $H$ for any $j = 1, \ldots, s-1$. Thus, we have $a(u_{i_j}) + a(v_{i_{j+1}}) \ge t > a(u_{i_{j+1}}) + a(v_{i_{j+1}})$, which implies $a(u_{i_j}) > a(u_{i_{j+1}})$ for each $j$. Consequently, the sequence of non-negative integers $a(u_{i_j})$ is strictly decreasing, yielding $s \le a(u_{i_1}) + 1 \le t$. It follows that $\ell(p) \le 2t-1$, bounding $\ell_0(M)$.

Next, suppose $p$ is the concatenation of two $M$-admissible paths via an edge $\{v_i, v_j\}$ for some $i < j$. The length of this path is $\ell(p) = 2d_i + 2d_j - 1$. By the previous argument, we know that $d_i \le a(u_i) + 1$ and $d_j \le a(u_j) + 1$. Since $\{v_i, v_j\} \in E(G)$ but $M$ is an induced matching in $H$, we have $\{v_i, v_j\} \notin E(H)$, which deduces $a(v_i) + a(v_j) \ge t$. Furthermore, by the admissibility of $H$, we have $a(u_i) + a(v_i) \le t-1$ and $a(u_j) + a(v_j) \le t-1$. Summing these lines gives:
$$a(u_i) + a(u_j) \le 2t - 2 - (a(v_i) + a(v_j)) \le 2t - 2 - t = t-2.$$
Thus, we obtain:
$$\ell(p) = 2d_i + 2d_j - 1 \le 2(a(u_i) + a(u_j)) + 3 \le 2(t-2) + 3 = 2t-1,$$
as required. This completes the proof.
\end{proof}

\begin{lem} \label{lem:2} 
Let $G$ be a simple graph and $M$ be an ordered matching of $G$ with free parameter set $A = \{u_1, \ldots, u_r\}$ and partner set $B = \{v_1, \ldots, v_r\}$. Let $t$ be a positive integer, and assume that $\ell(M) \le 2t-1$. Then $H = (V(G), M)$ is $t$-admissible.
\end{lem}

\begin{proof} 
For each $i$, let $d_i$ be the maximum number of matching edges in an $M$-admissible path starting with $v_i$. Since $\ell(M) \le 2t-1$, it follows that $d_i \le t$. We define the exponent vector $a \in \mathbb{N}^n$ as follows:
$$a(u_i) = d_i - 1, \quad a(v_i) = t - d_i, \quad a(v) = t \text{ if } v \notin M.$$ 
Note that $a(u_i) + a(v_i) = t-1 < t$. By definition, it suffices to prove that $a(u_i) + a(v_j) \ge t$ for every edge $\{u_i, v_j\}$ with $i < j$, and $a(v_i) + a(v_j) \ge t$ for every edge $\{v_i, v_j\}$ with $i < j$. Since $A$ is an independent set, $G$ contains no edges of the form $\{u_i, u_j\}$.  

\smallskip
\noindent \textbf{Case 1.} Assume that $\{u_i, v_j\}$ is an edge of $G$ with $i < j$. Then, any $M$-admissible path $p$ starting at $v_j$ can be extended to an $M$-admissible path $q = (u_i, v_i, v_j, p)$. Hence, $d_i \ge d_j + 1$, which implies that
$$a(u_i) + a(v_j) \ge (d_j + 1 - 1) + (t - d_j) = t.$$ 

\smallskip
\noindent \textbf{Case 2.} Assume that $\{v_i, v_j\}$ is an edge of $G$ with $i < j$. By Lemma~\ref{lem_length}, we have $\ell(M) \ge 2d_i - 1 + 2d_j - 1 + 1$. Since $\ell(M) \le 2t-1$, it follows that $2d_i + 2d_j - 1 \le 2t-1$, which simplifies to $d_i + d_j \le t$. Thus, 
$$a(v_i) + a(v_j) = (t - d_i) + (t - d_j) = 2t - (d_i + d_j) \ge t.$$
This completes the proof of the lemma.
\end{proof}

\begin{proof}[Proof of Theorem \ref{main}] Let $M$ be an ordered matching of $G$ with $\ell(M) \le 2t-1$ and $|M| = \alpha_t(G)$. By Lemma \ref{lem:2}, $H = (V(G),M)$ is $t$-admissible. By Theorem \ref{thm:admissible}, we deduce that 
$$\depth (S/J(G)^{(t)}) \le n - \reg (I(H)) = n - 1 - |M|.$$

Now, assume that $G$ is a forest. It remains to prove that for any $H \in \Adm_t(G)$, we have $\reg (I(H)) \le 1 +\alpha_t(G)$. By \cite[Theorem 2.18]{Z}, since $H$ is a subgraph of a forest, $H$ itself is a forest, we have $\reg (I(H)) = 1 + \nu(H)$. Now, let $M$ be an induced matching of $H$ with $|M| = \nu(H)$. By Lemma \ref{lem_ind_order}, $M$ is an ordered matching of $G$. By Lemma \ref{lem:1}, $\ell(M) \le 2t-1$. Hence, $|M| \le \alpha_t(G)$ by definition. That completes the proof of the theorem.
\end{proof}

When $G$ is a forest, we can drop the condition that $M$ is an ordered matching in Theorem~\ref{main} due to the following lemma. This result establishes a connection between the depth of a power of $J(G)$ and a combinatorial invariant of $G$.

\begin{lem}\label{lem_ord_forest}
Let $G$ be a forest and $M$ a matching of $G$. Then $M$ is an ordered matching.     
\end{lem}

\begin{proof}
Since the definition of an ordered matching only involves vertices appearing in $M$, we may assume without loss of generality that $M$ covers all vertices of $G$; that is, $G$ has a perfect matching. We proceed by induction on the number of vertices of $G$. 

The base case where $|V(G)| = 2$ (i.e., $G$ consists of a single edge) is trivial. Now, assume that $|V(G)| > 2$. Let $u$ be a leaf of $G$ and let $v$ be its unique neighbor. Because $M$ is a perfect matching, it necessarily follows that $\{u, v\} \in M$. Consider the smaller matching $M' = M \setminus \{\{u, v\}\}$ on the induced subgraph $G \setminus \{u, v\}$, which is also a forest. By the induction hypothesis, $M'$ is an ordered matching with free parameter set $A' = \{u_1, \ldots, u_{r-1}\}$ and partner set $B' = \{v_1, \ldots, v_{r-1}\}$. 

Now, let $A = A' \cup \{u\}$ and $B = B' \cup \{v\}$, and set $u_r = u$ and $v_r = v$. We claim that this extended ordering defines an ordered matching for $M$. Since $A'$ is an independent set and $u$ is a leaf whose only neighbor is $v \notin A$, it follows that $A$ is also an independent set. Furthermore, because $u_r = u$ is a leaf and has no edges connecting to any vertex in $B'$, there are no edges of the form $\{u_r, v_j\}$ for $j < r$. Hence, the requirements for an ordered matching are satisfied.
\end{proof}

\begin{cor}\label{thm:admissibletree} 
Let $T$ be a forest on $n$ vertices. Then
$$\operatorname{depth}(S/J(T)^{t}) = n - 1 - \alpha_t(T),$$
where 
$$\alpha_t(T) = \max \{ |M| \mid M \text{ is a matching of } T \text{ and } \ell(M) \le 2t - 1\}.$$
\end{cor}

\begin{proof}
By \cite[Theorem 5.1]{HHT}, we have $J(T)^{(t)} = J(T)^t$. The conclusion then follows immediately from Theorem \ref{main} and Lemma \ref{lem_ord_forest}.
\end{proof}

When $t$ is sufficiently large, $\alpha_t(G) = \nu_o(G)$ and the equality holds by the result of Constantinescu and Varbaro \cite{CV}. When $t = 1$, equality holds if and only if $\operatorname{reg}(I(G)) = 1 + \nu(G)$. Hence, if we expect the equality to hold for all $t$, then $G$ must satisfy this condition. A natural class of graphs satisfying this condition is the class of weakly chordal graphs. Computational evidence suggests the following conjecture.

\begin{conj} 
Let $G$ be a weakly chordal graph on $n$ vertices. Then 
$$\operatorname{depth}(S/J(G)^{(t)}) = n - 1 - \alpha_t(G)$$
for all $t \ge 1$.
\end{conj}

\end{document}